\documentclass[10pt,a4paper]{article}
\usepackage[english]{babel}
\usepackage{amsmath}
\usepackage{amssymb}
\usepackage{amsthm}
\usepackage[pdftex]{color,graphicx}
\usepackage[utf8]{inputenc}
\usepackage[T1]{fontenc}
\usepackage{enumerate}
\usepackage{mathscinet}

\usepackage{bbm}
\usepackage{bm}
\usepackage{hyperref}

\newtheorem{tw}{Theorem}[subsection]
\newtheorem{lm}[tw]{Lemma}
\newtheorem{wn}[tw]{Corollary}
\newtheorem{pr}[tw]{Proposition}
\newtheorem{hyp}{Conjecture}
\theoremstyle{definition}

\newtheorem{uw}[tw]{Remark}

\newcommand{\R}{\mathbb{R}}
\newcommand{\Z}{\mathbb{Z}}
\newcommand{\N}{\mathbb{N}}
\newcommand{\T}{\mathbb{T}}

\newcommand{\C}{\mathbb{C}}
\newcommand{\Q}{\mathbb{Q}}

\newcommand{\raz}{\mathbbm{1}}

\newcommand{\ov}{\overline}

\newcommand{\vep}{\varepsilon}
\newcommand{\mobius}{M\"obius }
\newcommand{\bmu}{\bm \mu}
\newcommand{\mob}{\bm \mu}
\newcommand{\ot}{\otimes}
\newcommand{\la}{\lambda}

\providecommand{\noopsort}[1]{} %potrzebne, gdy nazwiska z prefiksami!

\title{The M\"{o}bius function and continuous extensions of rotations}
\author{J. Ku\l{}aga-Przymus\thanks{Research supported by Narodowe Centrum Nauki grant DEC-2011/03/B/ST1/00407.}\and M. Lema\'{n}czyk\footnotemark[1]}

\begin{document}

\bibliographystyle{siam}

\maketitle
\begin{abstract}
Let $f\colon \T\to \R$ be of class $C^{1+\delta}$ for some $\delta>0$ and let $c\in\Z$. We show that for a generic $\alpha\in\R$, the extension $T_{c,f}\colon \T^2\to\T^2$ of the irrational rotation $Tx=x+\alpha$, given by $T_{c,f}(x,u)=(x+\alpha, u+cx+f(x))$ ($\bmod\  1$) satisfies Sarnak's conjecture.
\end{abstract}
\tableofcontents

\section{Introduction}\label{se:1}
Recall that the \mobius function $\bmu\colon \N\to \{-1,0,1\}$ is a multiplicative function\footnote{Recall that $\nu\colon \N\to\Z$ is said to be multiplicative if $\nu(mn)=\nu(m)\nu(n)$ for $n,m$ relatively prime.} defined as $\bmu(p_1\cdot\ldots\cdot p_k)=(-1)^k$ for  distinct prime numbers $p_j$, $\bmu(1)=1$ and 0 otherwise.
Its importance is reflected in the fact that the prime number theorem\footnote{Recall that the prime number theorem states that $\pi(x)=\frac{x}{\ln x}+{\rm o}(\frac{x}{\ln x})$, where $\pi(x)$ is the number of primes less than $x$.} is equivalent to the condition $\sum_{k\leq x}\bmu(k)={\rm o}(x)$ and the Riemann hypothesis is equivalent to the condition $\sum_{k\leq x}\bmu(k)={\rm O}_\vep(x^{\frac{1}{2}+\vep})$ for any $\vep>0$ (when $x\to\infty$). The \mobius function appears to behave rather randomly and this statement was formalized in the following conjecture of Sarnak:
\begin{hyp}[\cite{MR3014544}]\label{con:1}
Let $X$ be a compact metric space and let $T\colon X\to X$ be a homeomorphism of zero topological entropy. Let $x\in X$, $g\in C(X)$. Then
\begin{equation}\label{eq:zero}
\sum_{n\leq N}g(T^nx)\bmu(n)={\rm o}(N).
\end{equation}
\end{hyp}
Whenever condition~\eqref{eq:zero} is true for some $T$ for each $x\in X$ and each $g\in C(X)$, we say that \emph{Sarnak's conjecture holds} for $T$ or that $T$ is \emph{disjoint from the \mobius function}.\footnote{Notice that it suffices to show that~\eqref{eq:zero} holds for a linearly dense set of continuous functions to obtain disjointness with the \mobius function.}

Sarnak's conjecture is known to hold in several situations, including rotations~\cite{davenport}, nilsystems~\cite{MR2877066}, horocycle flows~\cite{MR2986954}, large class of rank one maps~\cite{MR3095150,MR3121731} and certain subclasses of dynamical systems generated by generalized Morse sequences~\cite{MR0239047}, including the classical Thue-Morse system:~\cite{Abdalaoui:2013rm,MR3043150,MR2207389,MR2981162,MR1912360,MR2680394} and the dynamical system generated by the Rudin-Shapiro sequence~\cite{MaRi2013}.

One of the most fruitful tools used for proving disjointness with the M\"obius function turns out to be the following orthogonality criterion of Katai-Bourgain-Sarnak-Ziegler (we will refer to it as KBSZ criterion).
\begin{tw}[\cite{MR836415,MR2986954}]
\label{tw:kryterium}
Let $F\colon \N\to \C$ be a bounded sequence. Suppose that
\begin{equation}\label{eq:kryterium}
\sum_{n\leq N}F(rn)\overline{F(sn)}={\rm o}(N)
\end{equation}
for any pair of sufficiently large primes $r\neq s$. Then
\begin{equation}\label{eq:sarnak}
\sum_{n\leq N}F(n)\nu(n)={\rm o}(N),
\end{equation}
for any multiplicative function $\nu$ with $|\nu|\leq 1$.
\end{tw}
In order to use this theorem for proving Conjecture~\ref{con:1} for a given homeomorphism $T\colon X\to X$, one takes
\begin{equation}\label{eq:zg}
F(n):=g(T^nx) \text{ for }n\in\Z, x\in X \text{ and }g\in C(X).
\end{equation}
Notice that the expression~\eqref{eq:kryterium} (for $F(n)$ given by~\eqref{eq:zg}) takes the form
$$
\frac{1}{N}\sum_{n\leq N}g\otimes \overline{g}\left((T^r\times T^s)^n(x,x) \right)=\int_{X^2} g\otimes \overline{g}\ d\left(\frac{1}{N}\sum_{n\leq N}\delta_{(T^r\times T^s)^n(x,x)}\right),
$$
where $x\in X$. It follows that we are interested in the limit measures $\rho=\lim_{k\to\infty}\frac{1}{N_k}\sum_{n\leq N_k}\delta_{(T^r\times T^s)^n(x,x)}$ which are clearly $T^r\times T^s$-invariant. Therefore, to prove disjointness of $T$ with the M\"obius function, it suffices to show that the following holds for $T$, for $r,s$ relatively prime:
\begin{enumerate}[(a)]
\item\label{numer:a}
The ergodic components of $T^r\times T^s$ are pairwise disjoint closed sets filling up the whole space.\footnote{The proof of the main result of the paper (Theorem~\ref{thm:main} below) says also that the ergodic decomposition  will be the same as the decomposition into minimal components which seems to be a fact of independent interest. In case of continuous compact group extensions the existence of the decomposition into minimal components is guaranteed by a result of Auslander~\cite{MR0164335} and Ellis~\cite{MR0123636} on distal systems.}
\item\label{numer:b}
The ergodic components are uniquely ergodic (this implies that all points are generic for $T^r\times T^s$ for relevant invariant measures).
\item\label{numer:c}
There exists a linearly dense set\footnote{Notice that we do not aim  to prove~\eqref{eq:kryterium} (for $F(n)$ given by~\eqref{eq:zg}) for each $g\in C(X)$ (and each $x\in X$) -- as a matter of fact~\eqref{eq:kryterium} (for $F(n)$ given by~\eqref{eq:zg}) is not satisfied for some continuous functions already for an irrational rotation; we provide examples in Appendix, see Proposition~\ref{p1}. Our aim is to prove~\eqref{eq:kryterium} for
a linearly dense set of $g\in C(X)$ (and each $x\in X$), as it implies that~\eqref{eq:sarnak} holds for each $g\in C(X)$.
} of continuous functions ${\mathcal{F}}\subset C(X)$ such that for each $g\in \mathcal{F}$, $g\circ T\neq g$, we have $\int_{X^2} g\otimes \overline{g}\ d\rho=0$ for any $T^r\times T^s$-ergodic measure $\rho$, whenever $r,s$ are sufficiently large.
\end{enumerate}

We will use the strategy~\eqref{numer:a},~\eqref{numer:b},~\eqref{numer:c} to study disjointness with the M\"obius function in the following setting.
Denote by $\T=\R/\Z=[0,1)$ the additive circle and consider
\begin{equation}\label{eq:1}
\T^2\ni (x,y)\mapsto T_{c,f}(x,y):=(x+\alpha,y+cx+f(x))\in\T^2.
\end{equation}
where $c\in\Z$  and $cx+f(x)$ is a lift of a continuous circle map, i.e.\ $f\colon \R\to \R$ is continuous, periodic of period~$1$, and $c$ is the degree of the map in question. In other words, we consider continuous case of the classical Anzai skew product extensions of a rotation on the circle~\cite{MR0040594}.

Liu and Sarnak in their recent paper~\cite{Liu:2013fk} proved the following.
\begin{tw}[\cite{Liu:2013fk}]\label{tw:liusar}
Assume that in~\eqref{eq:1}, $f\colon\R\to\R$ is an analytic  periodic function of period $1$. Assume additionally that $|\widehat{f}(m)| \gg e^{-\tau |m|}$ for some $\tau>0$. Then $T_{c,f}$ satisfies Conjecture~\ref{con:1}.
\end{tw}
The technical condition on the Fourier transform, namely $|\widehat{f}(m)|\gg e^{-\tau|m|}$, seems to be necessary for the methods of~\cite{Liu:2013fk} to work. On the other hand, there is no condition on $\alpha$. Moreover, for some $\alpha$s, the result is obtained using Theorem~\ref{tw:kryterium}. Under some additional assumptions, also a quantitative version (i.e.\ concerning the speed of convergence to zero in~\eqref{eq:zero}) of Theorem~\ref{tw:liusar} is proved in~\cite{Liu:2013fk}. This is achieved by treating the problem in a more direct way than applying Theorem~\ref{tw:kryterium}.

A natural question arises whether the strong assumptions on $f$ in Theorem~\ref{tw:liusar} can be relaxed. We do so in the main result of the paper (Theorem~\ref{thm:main} below) to obtain Sarnak's conjecture for each sufficiently smooth $f$  at the cost of reducing ``for each $\alpha$'' in Theorem~\ref{tw:liusar} to ``for a generic  $\alpha$''. Hence, our result can be viewed as a complementary to Liu-Sarnak's result.

\begin{tw}\label{thm:main}
Let $f\colon\R\to\R$ be a function of class $C^{1+\delta}$ for some $\delta>0$, periodic of period $1$. Let $c\in\Z$. Then for a generic set\footnote{The question of whether an analogous result to Theorem~\ref{thm:main} is true for $f$ which is only assumed to be continuous, remains open. We recall that under the continuity assumptions on $f$, even, it is open whether $f$ is not a quasi-coboundary for a generic set of $\alpha$.} of $\alpha$, the automorphism $T_{c,f}$ of $\T^2$ given by~\eqref{eq:1} satisfies Conjecture~\ref{con:1}.
\end{tw}

Before we give the proof of Theorem~\ref{thm:main}, we will first show that Conjecture~\ref{con:1} holds in the following two natural cases: the case of an arbitrary continuous extension of a rational rotation (see Proposition~\ref{p2} below)\footnote{Liu and Sarnak \cite{Liu:2013fk} prove such a result for $f$ smooth.}   and to get an independent proof of the purely affine case (i.e.\ when $f=0$) for each $\alpha$, first proved in a larger context in \cite{Liu:2013fk}.

\begin{tw}[\cite{Liu:2013fk}]\label{thm:affine}
For any $\alpha,\gamma\in\R$ and for any $c\in\Z$, the automorphism $(x,y)\mapsto (x+\alpha,cx+y+\gamma)$ satisfies Conjecture~\ref{con:1}.
\end{tw}

Let us now describe how we check conditions~\eqref{numer:a},~\eqref{numer:b} and~\eqref{numer:c}. Let $\alpha\not\in\Q$ and let $Tx=x+\alpha$. In either setting (purely affine or with a non-trivial perturbation) the base rotation $T^r\times T^s$ is the same. Its ergodic components are  obtained by taking the  partition of $\T^2$ into closed invariant sets $A_{c_1}=A_{c_1}^{r,s}:=\{(x,y+c_1)\in\T^2\colon sx=ry\}$, $c_1\in [0,\frac{1}{r})$. In fact, these sets are at the same time the minimal components of $T^r\times T^s$ and they are uniquely ergodic. Thus, we are interested in the action of $(T_{c,f})^r\times (T_{c,f})^s$ on the sets $I_{c_1}=I_{c_1}^{r,s}:=A_{c_1}\times \T^2$. It turns out that this is equivalent to dealing with extensions of $T$ by the following $\T^2$-valued cocycles:
$$
\psi_{c_1}(x)=(\psi^{(r)}(rx),\psi^{(s)}(sx+c_1)),
$$
where $\psi(x)=f(x)+cx$ and $c_1\in [0,\frac{1}{r})$. The cocycle $\psi_{c_1}$ is ergodic if and only if
$$
e^{2\pi i(A\psi^{(r)}(rx)+B\psi^{(s)}(sx+c_1))}\text{ is not a multiplicative coboundary}\footnote{We identify $\T$ with the multiplicative circle $\mathbb{S}^1=\{z\in\C\colon |z|=1\}$. We will use both, the additive and the multiplicative notation, whichever is more convenient for us at the moment. In particular,  $e^{2\pi i x}$ is to be understood multiplicatively and $x$ in the exponent is to be understood additively.}
$$
for $A,B\in \Z$, $A^2+B^2\neq 0$ (see Remark~\ref{uw:7}). This is the situation we aim for in course of the proof of Theorem~\ref{thm:main}. For a generic $\alpha$ we indeed obtain ergodicity of $\psi_{c_1}$ for all $c_1$ -- for the details see Corollary~\ref{wn:21}. Statement~\eqref{numer:b} follows from the fact that we deal with compact group extensions of rotations. Finally, we show that also~\eqref{numer:c} holds: given a non-trivial character $\chi\in\widehat{\T}^2$, we prove that for $r$ and $s$ relatively prime and large enough, the corresponding integrals of $\chi\otimes\overline{\chi}$ are zero. In case of Theorem~\ref{thm:affine}, the problem is in a sense more delicate -- some of the sets $I_{c_1}$ are too large to be the ergodic components  and they need to be partitioned further into smaller subsets. This refined partition will be however satisfying~\eqref{numer:a}. Condition~\eqref{numer:b} is proved in the same way as in Theorem~\ref{thm:main}. To prove that also~\eqref{numer:c} holds, we take again $\mathcal{F}=\widehat{\T}^2$.

For reader's convenience, we added Appendix collecting some necessary facts concerning cocycles.

\section{Results}
\subsection{On the strategy of the proofs}
Our approach to proving disjointness with the \mobius function was described in~\eqref{numer:a},~\eqref{numer:b} and~\eqref{numer:c}. We will now make some more comments on this method. Recall that in view of Theorem~\ref{tw:kryterium}, for a linearly dense set $\mathcal{F}$ of $g\in C(X)$ and each $x\in X$, what we want to prove is
\begin{equation}\label{e1}
\frac1N\sum_{n\leq N} g(T^{rn}x)\overline{g(T^{sn}x)}\to 0 \mbox{ as } N\to\infty
\end{equation}
for distinct, sufficiently large prime numbers $r,s$. When this is realized through \eqref{numer:a}, \eqref{numer:b} and~\eqref{numer:c}, we prove more. Namely, for each $g\in \mathcal{F}$ and for sufficiently large primes $r\neq s$, the following holds for all $x,y\in X$:
\begin{equation}\label{7a}
\frac{1}{N}\sum_{n\leq N}g(T^{rn}x)\overline{g(T^{sn}y)}\to 0\text{ as }N\to \infty
\end{equation}
(the condition on $r,s$ is independent of the choice of $x$ and $y$).

\begin{uw}\label{uw:13}
\noindent
\begin{enumerate}[(i)]
\item\label{rmi}
In view of the above discussion, in order to prove that Sarnak conjecture holds for $T$, it suffices to check conditions~\eqref{numer:a},~\eqref{numer:b} and~\eqref{numer:c} for $T^k$ for some $k\geq 1$. Indeed, by applying~\eqref{7a} to $T^k$, we obtain
\begin{multline*}
\sum_{n\leq N}g(T^{r(kn+j)}x)\overline{g(T^{s(kn+j)}x)}=\\
=\sum_{n\leq N}g((T^k)^{rn}T^{rj}x)\overline{g((T^k)^{sn}T^{sj}x)}={\rm o}(N)\text{ for }0\leq j<k,
\end{multline*}
which implies that~\eqref{e1}  holds for $T$.
\item
Notice that whenever the conditions~\eqref{numer:a},~\eqref{numer:b} and~\eqref{numer:c}  are satisfied for some homeomorphism $T$ then they are also satisfied for $T^{-1}$. However,
\item
It is unclear how to prove directly that if $T^{k}$ for some $k\in\Z\setminus\{0\}$ is disjoint from $\bmu$ then also $T$ is disjoint from $\bmu$, or even to show that if the assumptions of the KBSZ criterion are satisfied for $T^{k}$ for some $k\in\Z\setminus\{0\}$ then they are satisfied for $T$.
\end{enumerate}
\end{uw}

Let now $T_\varphi\colon X\times \T\to X\times \T$ be a continuous circle group extension of a homeomorphism $T\colon X\to X$ by $\varphi\colon X\to \T$.
\begin{uw}\label{uw:14}
Suppose that $T_\varphi$ satisfies Sarnak's conjecture. Then for any $k\geq 1$ also $T_{k\varphi}$ satisfies Sarnak's conjecture as it is a topological factor of $T_\varphi.$
\end{uw}

In the case of affine extensions of rotations, Remark~\ref{uw:13}~\eqref{rmi} and Remark~\ref{uw:14} are complementary in the following sense. Let $T^{(\alpha)}$ stand for the rotation $T^{(\alpha)}x=x+\alpha$ and let $\psi(x)=x$. Then $(T^{(\frac{\alpha}{k})}_\psi)^k=T^{(\alpha)}_{k\psi+\frac{k-1}{2}\alpha}$ and, by Remark~\ref{uw:13}~\eqref{rmi}, the following implication holds:
$$
\text{\eqref{numer:a},~\eqref{numer:b} and~\eqref{numer:c} hold for }T^{(\alpha)}_{k\psi+\frac{k-1}{2}\alpha}\Rightarrow T^{(\frac{\alpha}{k})}_\psi \text{ satisfies Conjecture~\ref{con:1}}.
$$

\subsection{General remarks}
From now on, our assumption will be that
\begin{equation*}
\text{$r,s\geq 3$ are odd and relatively prime.}
\end{equation*}
Let $\alpha\not\in\Q$ and denote by $T\colon \T\to \T$ the irrational rotation $Tx=x+\alpha$. For $c_1\in [0,\frac{1}{r})$ let
\begin{equation}\label{eq:i}
I_{c_1}=I_{c_1}^{r,s}:=A_{c_1}\times \T^2\text{, where }A_{c_1}=A_{c_1}^{r,s}:=\{(x,y+c_1)\in\T^2\colon sx=ry\}.\footnote{Whenever $r,s$ are fixed, we will write $I_{c_1}$ and $A_{c_1}$. For $r,s$ varying, we will use $I_{c_1}^{r,s}$ and $A_{c_1}^{r,s}$.}
\end{equation}

\begin{lm}\label{lm:3.2.1}
The decomposition of $\T^2$ into minimal components of $T^r\times T^s$ consists of sets $A_{c_1}$, $c_1\in [0,\frac{1}{r})$. It is the same as the ergodic decomposition.
Moreover, $(T^r\times T^s)|_{A_{c_1}}$ is topologically isomorphic to $T$. The isomorphism is given by
\begin{equation}\label{eq:wzorw}
W=W_{c_1}\colon A_{c_1}\to\T,\ W(x,y+c_1)=ax+by,
\end{equation}
where $a,b\in\Z$ are such that $ar+bs=1$.
\end{lm}
\begin{proof}
Notice first that the sets $A_{c_1}$ are closed and invariant under $T^r\times T^s$ and $\cup_{c_1\in[0,\frac{1}{r})}A_{c_1}=\T^2$. Fix $c_1\in[0,\frac{1}{r})$, let $a,b\in\Z$ be such that $ar+bs=1$ and let $W$ be given by~\eqref{eq:wzorw}. Then $W\circ (T^r\times T^s)|_{A_{c_1}}=T\circ W$. For $(x,y+c_1)\in A_{c_1}$, we have $r(ax+by)=x$ and $s(ax+by)=y$. Therefore, $W$ is bijective. Moreover, $W$ preserves the measure, as rotations are uniquely ergodic.
\end{proof}

For a measurable function $\psi\colon \T\to\T$, let $\Psi\colon \T^2\to\T^2$ be given by
\begin{equation}\label{eq:zmiana}
\Psi(x,y)=(\psi^{(r)}(x),\psi^{(s)}(y)).
\end{equation}
Then clearly the automorphism $(T_\psi)^r\times(T_\psi)^s$ is topologically isomorphic to $(T^r\times T^s)_{\Psi}$.
\begin{lm}\label{lm:7}
For $c_1\in [0,\frac{1}{r})$ the sets $I_{c_1}$ are invariant under $(T^r\times T^s)_\Psi$. Moreover, $(T^r\times T^s)_\Psi|_{I_{c_1}}$ is topologically isomorphic to $T_{\psi_{c_1}}$, where $\psi_{c_1}(x)=(\psi^{(r)}(rx),\psi^{(s)}(sx+c_1))$. The isomorphism is given by
\begin{equation}\label{eq:wzorv}
V=V_{c_1}\colon I_{c_1}\to\T^3,\ V(x,y+c_1,u,v)=(ax+by,u,v),
\end{equation}
with $a,b\in\Z$ such that $ar+bs=1$.
\end{lm}
\begin{proof}
Fix $c_1\in[0,\frac{1}{r})$, let $a,b\in\Z$ be such that $ar+bs=1$ and let $V$ be given by~\eqref{eq:wzorv}. Then
\begin{align*}
V\circ (T^r\times T^s)_{\Psi}&(x,y+c_1,u,v)\\
&=V(x+r\alpha,y+s\alpha+c_1,u+\psi^{(r)}(x),v+\psi^{(s)}(y+c_1))\\
&=(ax+by+\alpha,u+\psi^{(r)}(x),v+\psi^{(s)}(y+c_1))
\end{align*}
and
\begin{align*}
T_{\psi_{c_1}}\circ &V(x,y+c_1,u,v)=T_{\psi_{c_1}}(ax+by,u,v)\\
&=(ax+by+\alpha,u+\psi^{(r)}(r(ax+by)),v+\psi^{(s)}(s(ax+by)+c_1)).
\end{align*}
Moreover, $r(ax+by)=x\text{ and }s(ax+by)=y$, which completes the proof.
\end{proof}

\begin{uw}\label{odwrv}
The inverse of $V_{c_1}\colon I_{c_1}\to\T^3$ is given by
$$
V_{c_1}^{-1}(t,u,v)=(rt,st+c_1,u,v).
$$
Moreover, $art+bst=t$ for $a,b\in\Z$ such that $ar+bs=1$.
\end{uw}

\subsection{Continuous extensions of rational rotations}
We will now show that all continuous extensions of rational rotations satisfy Sarnak's conjecture.

\begin{lm}\label{l1} Assume that $S(x,y)=(x,f(x)+y)$ with $f\colon\T\to\T$ continuous. Then for each $(x_i,y_i)\in\T^2$, $i=1,2$, $\chi\in\widehat{\T^2}$ we have
\begin{equation}\label{e7}
\frac1N\sum_{n\leq N} \chi(S^{rn}(x_1,y_1))\ov{\chi(S^{sn}(x_2,y_2))}\to 0\;\;\mbox{as}\;\;N\to\infty\end{equation}
for sufficiently large prime numbers $r\neq s$, whenever $\chi\neq \eta\ot \mathbbm{1}_{\T}$ ($\eta\in\widehat{\T}$) and $f(x_1)\notin\Q$ or $f(x_2)\notin\Q$.\end{lm}
\begin{proof}
Note that for each $m\geq1$, $$S^m(x,y)=
(x,mf(x)+y).$$
Let $\chi(x,y)=e^{2\pi i(ax+by)}$ for some $a,b\in\Z$, with $b\neq0$ by assumption. Hence
\begin{multline*}
\frac1N\sum_{n\leq N} \chi(S^{rn}(x_1,y_1))\ov{\chi(S^{sn}(x_2,y_2))}\\
=e^{2\pi i(a(x_1-x_2)+b(y_1-y_2))}\frac1N\sum_{n\leq N}e^{2\pi ib(rf(x_1)-sf(x_2))n}.
\end{multline*}
If exactly one of the numbers $f(x_1)$ and $f(x_2)$ is irrational then the result follows from Weyl's criterion. If both $f(x_1)$ and $f(x_2)$ are irrational then there is at most one pair $(r,s)$ of relatively prime numbers such that $rf(x_1)-sf(x_2)\in \Q$ and we can again make use of Weyl's criterion, this time for $r,s$ sufficiently large.
\end{proof}

\begin{uw}\label{r2}
Notice that the above proof says more. Namely, the convergence in~\eqref{e7} does not depend on $y_1,y_2$.\end{uw}

\begin{pr}\label{p2}
Each continuous extension $R(x,y)=(x+\frac pq,f(x)+y)$ of a rational rotation $x\mapsto x+p/q$ satisfies Sarnak's conjecture.
\end{pr}
\begin{proof}We need to check~\eqref{eq:zero} only for $F=\chi\in\widehat{\T^2}$.
First, note that \begin{equation}\label{e9}R^q(x,y)=(x,f_q(x)+y),\end{equation} where
$f_q(x)=f(x)+f(x+\frac1q)+\ldots+f(x+\frac{q-1}q)$.
Given $n\geq1$, we take $n'$ such that $n=qn'+j$ with $0\leq j<q$. Then,
for each $\chi\in\widehat{\T^2}$, we have
$$
\chi(R^{rn}(x,y))\ov{\chi(R^{sn}(x,y))}=
\chi(R^{qrn'}(R^{rj}(x,y)))\ov{\chi(R^{qsn'}(R^{sj}(x,y)))},$$
where the first coordinates of the points $R^{rj}(x,y),R^{rj}(x,y)$ belong to the finite set $\{x,x+\frac1q,\ldots,x+\frac{q-1}q\}$ (hence do not depend on $r,s$). Hence, to show
$$
\frac1N\sum_{n\leq N} \chi(R^{rn}(x,y))\ov{\chi(R^{sn}(x,y))}\to0$$
for sufficiently large prime numbers $r\neq s$, we need to show that
$$
\frac1N\sum_{n'\leq N/q} \chi(R^{qrn'}(x_1,\ast))\ov{\chi(R^{qsn'}(x_2,\ast))}\to0,$$
for $x_1,x_2\in\{x,x+1/q,\ldots,x+\frac{q-1}q\}$.
This is the case by Lemma~\ref{l1}, Remark~\ref{r2} and~\eqref{e9}, provided that $\chi\neq\eta\ot \mathbbm{1}_{\T}$ and $f_q(x_1)\not\in\Q$ or $f_q(x_2)\notin\Q$.

If $\chi=\eta\ot \mathbbm{1}_{\T}$, \eqref{eq:zero} follows from Sarnak's conjecture for finite systems. Suppose now that $f_q(x+jrp/q),f_q(x+jsp/q)\in\Q$. This is possible only if $f_q(x)\in\Q$
since $f_q(\cdot)$ is constant on the orbit of $x$ under the rotation $x\mapsto x+\frac{p}q$. Moreover, if $n=qn'+j$ with $0\leq j<q$ then
\begin{equation}\label{e10}
f^{(n)}(x):=\sum_{i=0}^{n-1}f(x+ip/q)=f^{(j)}(x)+f^{(qn')}(x+jp/q)=
f^{(j)}(x)+n'f_q(x).
\end{equation}
It follows that
\begin{align*}
\frac1N&\sum_{n\leq N}\chi(R^n(x,y))\mob(n)\\
&=\sum_{j=0}^{q-1}\frac1N\sum_{n'\leq N/q}\chi(x+(qn'+j)p/q,f^{(j)}(x)+n'f_q(x)+y)\mob(qn'+j)\\
&=\sum_{j=0}^{q-1}\frac1N\sum_{n'\leq N/q}\chi(x+jp/q,f^{(j)}(x)+n'f_q(x)+y))\mob(qn'+j)\\
&=\sum_{j=0}^{q-1}e^{2\pi i(a(x+jp/q)+bf^{(j)}(x))}\frac1N\sum_{n'\leq N/q}e^{2\pi ibf_q(x)n'}\mob(qn'+j)\\
&=\sum_{j=0}^{q-1}e^{2\pi i(a(x+jp/q)+bf^{(j)}(x))}\frac1N\sum_{n'\leq N/q}e^{2\pi ib\frac{c}{d}n'}\mob(qn'+j),
\end{align*}
%$$
%\frac1N\sum_{n\leq N}\chi(R^n(x,y))\mob(n)=
%$$
%$$
%\sum_{j=0}^{q-1}\frac1N\sum_{n'\leq N/q}\chi(x+(qn'+j)p/q,f^{(j)}(x)+n'f_q(x)+y)\mob(qn'+j)=
%$$
%$$
%\sum_{j=0}^{q-1}\frac1N\sum_{n'\leq N/q}\chi(x+jp/q,f^{(j)}(x)+n'f_q(x)+y))\mob(qn'+j)=
%$$
%$$
%\sum_{j=0}^{q-1}e^{2\pi i(a(x+jp/q)+bf^{(j)}(x))}\frac1N\sum_{n'\leq N/q}e^{2\pi ibf_q(x)n'}\mob(qn'+j)=
%$$
%$$
%\sum_{j=0}^{q-1}e^{2\pi i(a(x+jp/q)+bf^{(j)}(x))}\frac1N\sum_{n'\leq N/q}e^{2\pi ib\frac{c}{d}n'}\mob(qn'+j),
%$$
where $\chi(x,y)=e^{2\pi i(ax+by)}$, $f_q(x)=\frac{c}{d}$ with $a,b,c,d\in\Z$, $b\neq 0$ and $d>0$. By seting $n'=dn^{\prime\prime}+k$ with $0\leq k<d$, we obtain $qn'+j=qdn^{\prime\prime}+(qk+j)$ and rewriting the sum to the form
$$
\sum_{j=0}^{q-1}\sum_{k=0}^{d-1} \frac1N\sum_{n^{\prime\prime}\leq N/(dq)} A_{j,k}\mob(qdn^{\prime\prime}+kq+j),
$$
the result again follows from Sarnak's conjecture for finite systems.
\end{proof}

\subsection{Affine case}
Recall that we are interested in the disjointness with the \mobius function of $(x,y)\mapsto (x+\alpha,cx+y+\gamma)$, where $c\in\Z$ and $\alpha,\gamma\in\R$.

\begin{uw}
If $c\neq 0$, it follows immediately from Remark~\ref{uw:14} that instead of $(x,y)\mapsto (x+\alpha,y+cx+\gamma)$, we can consider $(x,y)\mapsto (x+\alpha,y+x+\frac{\gamma}{c}$). By Proposition~\ref{p2}, this shows that it suffices to consider only the cases $c=1$ and $c=0$, with $\alpha\not\in\Q$.
\end{uw}

\subsubsection{Case $c=1$, $\alpha \not\in\Q$}\label{se:332}
We will now deal with $(x,y)\mapsto(x+\alpha,y+x+\gamma)$, i.e.\ with $T_\psi$, where $\psi(x)=x+\gamma$. Let $\Psi$ be given by~\eqref{eq:zmiana}, i.e.
$$
\Psi(x,y)=\left(rx+\frac{(r-1)r}{2}\alpha+r\gamma,sy+\frac{(s-1)s}{2}\alpha+s\gamma\right).
$$
Recall that given $r,s\in\N$ (odd and relatively prime) and $c_1\in [0,\frac{1}{r})$, we have
$$
I_{c_1}=I_{c_1}^{r,s}=\left\{(x,y+c_1,u,v)\in \T^4\colon sx=ry\right\}.
$$
For $c_1$ such that $rs((s-r)\gamma-rc_1)\in\alpha\Q+\Q$ and $c_2\in[0,\frac{1}{r^2})$, define
\begin{align*}\label{eq:jcc}
J_{c_1,c_2}=J_{c_1,c_2}^{r,s}:=&\Big\{(x,y+c_1,u,v+c_2)\in\T^4 \colon sx=ry\text{ and }\\
&\left.l_0s^2u=l_0r^2v+\left(l_0rs\frac{r-s}{2}-k_0\right)(ax+by) \right\},
\end{align*}
where $l_0=l_0^{r,s,c_1}$ is the smallest positive integer such that
\begin{equation}\label{eq:l}
l_0rs((s-r)\gamma-rc_1)\in \alpha\Z+\Z
\end{equation}
and $k_0=k_0^{r,s,c_1}\in\Z$ is such that
\begin{equation}\label{eq:k}
l_0rs((s-r)\gamma-rc_1)-k_0\alpha\in\Z.
\end{equation}

\begin{lm}\label{lm:10}
For $c_1\in[0,\frac{1}{r})$ the homeomorphism $T_{\psi_{c_1}}\colon \T^3\to\T^3$, where $\psi_{c_1}(x)=(\psi^{(r)}(rx),\psi^{(s)}(sx+c_1))$, is topologically isomorphic to $T_{\varphi_{c_1}}\colon\T^3\to\T^3$, where $\varphi_{c_1}\colon \T\to\T^2$ is given by $\varphi_{c_1}(x)=(r^2x+r\gamma,s^2x+sc_1+s\gamma)$. The isomorphism is given by
\begin{equation}\label{eq:wzoru}
U=U_{c_1}\colon \T^3\to\T^3,\ U(x,u,v)=\left(x,u-\frac{r(r-1)}{2}x,v-\frac{s(s-1)}{2}x\right).
\end{equation}
\end{lm}
\begin{proof}
We have
\begin{multline*}
\psi_{c_1}(x)=\left(\psi^{(r)}(rx),\psi^{(s)}(sx+c_1)\right)\\
=\left(r^2x+\frac{r(r-1)}{2}\alpha+r\gamma,s(sx+c_1)+\frac{s(s-1)}{2}\alpha+s\gamma\right).
\end{multline*}
Since for $\theta(x)=\left(-\frac{r(r-1)}{2}x,-\frac{s(s-1)}{2}x \right)$ we have
$$
\left(\frac{r(r-1)}{2}\alpha,\frac{s(s-1)}{2}\alpha \right)=\theta(x)-\theta(x+\alpha),
$$
it follows that $U$ given by~\eqref{eq:wzoru} is indeed the required isomorphism. Notice that $\theta$ is continuous, whence $U$ is a homeomorphism.
\end{proof}

\begin{pr}\label{pr:11}
The decomposition of $\T^4$ into minimal components for $(T^r\times T^s)_\Psi$ is the same as the decomposition into ergodic components. It consists of sets of the form $I_{c_1}$ for $c_1\not\in\alpha\Q+\Q$ and $J_{c_1,c_2}$ for $c_1\in\alpha\Q+\Q$, where $c_1\in[0,\frac{1}{r})$, $c_2\in [0,\frac{1}{r^2})$. Moreover, on each such component, $(T^r\times T^s)_\Psi$ is uniquely ergodic. In particular, each point in $\T^4$ is generic (for a relevant invariant measure).
\end{pr}
\begin{proof}
In view of Lemma~\ref{lm:10}, instead of $(T^r\times T^s)_{\Psi}$, we may consider $T_{\varphi_{c_1}}$, with $\varphi_{c_1}(x)=(r^2x+r\gamma,s^2+sc_1+s\gamma)$. By Remark~\ref{uw:9}, Proposition~\ref{pr:5} and Proposition~\ref{tw:top}, for the first part of the assertion, we need to show that the equation
\begin{equation}\label{eq:charakt}
\chi\circ \varphi_{c_1}=\xi-\xi\circ T.
\end{equation}
has a measurable solution $\xi\colon \T\to\T$ for $\chi(u,v)=Au+Bv$ for $A,B\in\Z$ with $A^2+B^2\neq 0$ if and only if it has a continuous one. We have
$$
\chi\circ \varphi_{c_1}(x)=\chi(r^2x+r\gamma,s^2x+sc_1+s\gamma)=
(Ar^2+Bs^2)x+(Ar+Bs)\gamma+Bsc_1.
$$
By~\cite{MR0040594}, $\chi\circ \varphi_{c_1}$ can be a measurable coboundary only if
\begin{equation}\label{eq:AB}
Ar^2+Bs^2=0\text{ and }(Ar+Bs)\gamma+Bsc_1\in\alpha\Z+\Z.
\end{equation}
Then the solution to~\eqref{eq:charakt} is given by $\xi(x)=-kx$, where $k\in\Z$ is such that $(Ar+Bs)\gamma+Bsc_1-k\alpha\in\Z$. In particular, all measurable solutions to~\eqref{eq:charakt} are continuous. It follows by Remark~\ref{uw:9} that the decomposition into minimal components for $(T^r\times T^s)_\Psi$ is the same as the decomposition into ergodic components.

We will now describe the ergodic (i.e.\ minimal) components of each $T_{\varphi_{c_1}}$ and show that they are uniquely ergodic. Suppose that $\chi\circ \varphi_{c_1}$ is a coboundary. It follows from~\eqref{eq:AB} (recall that $r\neq s$ are coprime) that $r^2|B$ and $s^2|A$, and therefore $A=A's^2$ and $B=-A'r^2$ for some $A'\in\Z$. Hence, the second part of condition~\eqref{eq:AB} takes the form
\begin{equation}\label{eq:AB1}
A'rs((s-r)\gamma-rc_1)\in \alpha\Z+\Z.
\end{equation}
Having this in mind, we consider two cases:
\begin{enumerate}[(i)]
\item\label{numerek:1}
$(s-r)\gamma-rc_1\not \in \Q\alpha+\Q$,
\item\label{numerek:2}
$(s-r)\gamma-rc_1\in \Q\alpha+\Q$.
\end{enumerate}
In case~\eqref{numerek:1}, it follows immediately from the first part of the proof that $(T^r\times T^s)_{\Psi}|_{I_{c_1}}$ is ergodic and minimal. Unique ergodicity follows from Proposition~\ref{pr:fur}.

We consider now case~\eqref{numerek:2}. We will describe characters $\chi$, such that~\eqref{eq:charakt} has a (measurable and continuous) solution. In view of and \eqref{eq:l} and~\eqref{eq:AB1}, $A'=al_0$ for some $a\in \Z$. Therefore, a measurable solution $\xi\colon\T\to\T$ to~\eqref{eq:charakt} exists precisely for the characters $\chi$ of $\T^2$ of the form
\begin{equation}\label{char}
\chi(u,v)=al_0s^2u-al_0r^2v \text{ for }a\in\Z.
\end{equation}
Denote the set of such characters by $\Gamma=\Gamma_{top}$. It is easy to see that
$$
\text{ann }\Gamma=\left\{(u,v)\in \T^2\colon l_0s^2 u=l_0r^2v\right\}.
$$
We claim that
$$
\widetilde{J}_{c_1,c_2}^{r,s}:=\left\{(t,u,v+c_2)\in\T^3\colon l_0s^2u=l_0r^2v-k_0t\right\}\text{ for $c_2\in [0,1/r^2)$}
$$
are minimal components of $T_{\varphi_{c_1}}$. Indeed:
\begin{itemize}
	\item each $\widetilde{J}_{c_1,c_2}^{r,s}$ is closed,
	\item by~\eqref{eq:k}, each $\widetilde{J}_{c_1,c_2}^{r,s}$ is $T_{\varphi_{c_1}}$-invariant,
	\item $\bigcup_{c_2}\widetilde{J}_{c_1,c_2}^{r,s}=\T^3$ (notice that the projection of $\widetilde{J}_{c_1,c_2}^{r,s}$ onto the first two coordinates is equal to $\T^2$),
	\item $(u,v)\in \text{ann }\Gamma$ if and only if $\widetilde{J}_{c_1,c_2}^{r,s}+(0,u,v)=\widetilde{J}_{c_1,c_2}^{r,s}$
\end{itemize}
(see Proposition~\ref{tw:top}). Unique ergodicity follows, as in the previous case, from Proposition~\ref{pr:fur}.

To find the minimal component corresponding to each of the sets $\widetilde{J}_{c_1,c_2}^{r,s}$ described above, we need to find the preimage of $\widetilde{J}_{c_1,c_2}^{r,s}$ via $U_{c_1}\circ V_{c_1}$. We have
\begin{align*}
U_{c_1}^{-1}&(\widetilde{J}_{c_1,c_2}^{r,s})=\\
&=\left\{ (t,u,v+c_2)\colon l_0s^2\left(u-\frac{r(r-1)}{2}t\right)=l_0r^2\left(v-\frac{s(s-1)}{2}t\right)-k_0t \right\}\\
&=\left\{(t,u,v+c_2)\colon l_0s^2u=l_0r^2v+\left(l_0rs\frac{r-s}{2}-k_0\right)t \right\},
\end{align*}
whence, by Remark~\ref{odwrv},
\begin{align*}
(U_{c_1}V_{c_1})^{-1}(\widetilde{J}_{c_1,c_2}^{r,s})=&\left\{(rt,st+c_1,u,v+c_2)\colon\right.\\
&\left. l_0s^2u=l_0r^2v+\left(l_0rs\frac{r-s}{2}-k_0\right)t \right\}\\
=&\Big\{(x,y+c_1,u,v+c_2) \colon sx=ry\text{ and }\\
&\left.l_0s^2u=l_0r^2v+\left(l_0rs\frac{r-s}{2}-k_0\right)(ax+by) \right\}.
\end{align*}
Therefore
$(U_{c_1}V_{c_1})^{-1}\left(\widetilde{J}_{c_1,c_2}^{r,s}\right)=J_{c_1,c_2}^{r,s}$,
which completes the proof.
\end{proof}

\begin{uw}\label{uw:18}
The sets $I_{c_1}^{r,s}$ are translates of $I_{0}^{r,s}$, which is a subgroup of $\T^4$. Therefore, the conditional measures given by the ergodic decomposition of $(T^r\times T^s)_\Psi$, which are supported on the sets $I_{c_1}$ for $c_1$ such that $(s-r)\gamma-rc_1\not \in \Q\alpha+\Q$, are translates of Haar measure on $I_{0}^{r,s}$.

For $c_1$ with $(s-r)\gamma-rc_1\in \Q\alpha+\Q$, the sets $\widetilde{J}_{c_1,c_2}^{r,s}$ are translates of $\widetilde{J}_{c_1,0}^{r,s}$, which is a subgroup of $\T^3$. As before, the conditional measures given by the ergodic decomposition of $T_{\varphi_{c_1}}$, are translates of Haar measure on $\widetilde{J}_{c_1,0}^{r,s}$. Notice that $V_{c_1}$ carries a translate of Haar measure to Haar measure. Therefore, also the conditional measures given by the ergodic decomposition of $(T^r\times T^s)_\Psi$, which are supported on the sets $J_{c_1,c_2}^{r,s}$, are translates of the corresponding Haar measures.
\end{uw}

\begin{uw}\label{uw:19}
	Let $G$ be a compact Abelian group. Let $\chi\in \widehat{G}$ and let $H=\ker\chi$. Then $\chi$ is constant on each coset of $H$. Moreover, if $x+H\neq y+H$ then $\chi(x+H)\neq \chi(y+		H)$.
\end{uw}
\begin{uw}\label{uw:20}
	Let $\chi\in \widehat{G}$ and let $H\subset G$ be a subgroup. Then the integral of $\chi$ with respect to Haar measure on each coset of $H$ is zero if and only if $\chi|_H$ is not constant.
\end{uw}

\begin{lm}\label{lm:grupy}
	Let $G$ be a compact metric Abelian group. Let $\chi\in\widehat{G}$. Then there exists $\delta>0$ such that whenever $\{y_1,\dots,y_n\}\subset G$ is a $\delta$-net and $\chi(y_i)=1$ for 	$i=1,\dots,n$ then $\chi\equiv 1$.
\end{lm}
\begin{proof}
	Let $d$ stand for the metric on $G$. Since $\chi$ is uniformly continuous, there exists $\delta>0$ such that $d(x,y)<\delta$ implies $|\chi(x)-\chi(y)|<\frac{1}{4}$. Since $\{y_1,\dots, y_n\}$ is a 	 $\delta$-net, it follows that $|\chi(G)-1|<\frac{1}{2}$. This is however possible only when $\chi\equiv 1$, as $\chi(G)$ is a closed subgroup of $\mathbb{S}^1$.
\end{proof}

\begin{lm}\label{lm:12}
	Let $\chi\in\widehat{\T}^4$ be a non-trivial character. If $r,s\in\N$ are large enough then $\chi|_{I_{c_1}}$ is not constant for $c_1\in [0,\frac{1}{r})$.
\end{lm}
\begin{proof}
	Fix $1\not\equiv\chi\in\widehat{\T}^4$ and let $\delta>0$ be as in Lemma~\ref{lm:grupy}. In view of Lemma~\ref{lm:grupy}, we need to show that whenever $r,s$ are large enough then 		there exists a $\delta$-net of $\T^4$ contained in $I_{c_1}^{r,s}$ for $c_1\in [0,\frac{1}{r})$. Since $I_{c_1}^{r,s}$ is a translate of $I_0^{r,s}$ and the third and fourth coordinate in $I_{0}$ is arbitrary, it suffices to prove that for $r,s$ sufficiently large we can always find a $\delta$-net of $\T^2$ contained in the set $A_{0}^{r,s}$. To complete the proof it suffices to consider sets of the form
	$
	\left\{\left(\frac{i}{s},\frac{j}{r}\right)\colon 0\leq i<s,0\leq j<r\right\}.
	$
\end{proof}

\begin{lm}\label{lm:13}
Let $\chi\in\widehat{\T}^4$ be a non-trivial character. If $r,s\in\N$ are large enough then $\chi|_{J_{c_1,c_2}^{r,s}}$ is not constant for $c_1\in [0,\frac{1}{r})$ such that $c_1
\in\alpha\Q+\Q$ and $c_2\in [0,\frac{1}{r^2})$.
\end{lm}
\begin{proof}
Fix $1\not\equiv\chi\in\widehat{\T}^4$ and let $\delta>0$ be as in Lemma~\ref{lm:grupy}. Since $J_{c_1,c_2}^{r,s}$ is a translation of $J_{c_1,0}^{r,s}$, by  Lemma~\ref{lm:grupy}, it suffices to prove that for $r,s$ sufficiently large there exists a $\delta$-net of $\T^4$ in $J_{c_1,0}^{r,s}$.
	
Notice first that the projection of $J_{c_1,0}^{r,s}$ onto the first two coordinates is equal to $A_{c_1}^{r,s}$. Indeed, for any $z\in\T$ the equation $l_0s^2u=l_0r^2v+z$ has a solution $(u,v)$ as $\T$ is an infinitely divisible group. By the proof of the previous lemma, we can find a $\delta$-net of $\T^2$
$$
\left\{(x_i,y_j)\colon 0\leq i,j<n\right\}\subset A_{c_1}^{r,s}.
$$
Moreover, by the infinite divisibility of $\T$, for each $(i,j)$ there exist $u_{i,j},v_{i,j}$ such that $(x_i,y_j,u_{i,j},v_{i,j})\in J_{c_1,0}^{r,s}$. Moreover, for $0\leq a,b<l_0s^2$
$$
\left(x_i,y_j,u_{i,j}+\frac{a}{l_0s^2},v_{i,j}+\frac{b}{l_0s^2}\right)\in J_{c_1,0}^{r,s}.
$$
Therefore, whenever $r,s$ are sufficiently large, the set
$$
\left\{\left(x_i,y_j,u_{i,j}+\frac{a}{l_0s^2},v_{i,j}+\frac{b}{l_0s^2} \right)\colon 0\leq i,j<n, 0\leq a,b<l_0s^2 \right\}
$$
is the required $\delta$-net.\footnote{Notice that since $l_0\geq 1$, the condition on $r,s$ is independent of $c_1$.} This completes the proof.
\end{proof}

%As an immediate consequence of Proposition~\ref{pr:11}, Lemma~\ref{lm:12} and Lemma~\ref{lm:13}, we obtain the following.
\begin{pr}\label{pr:14}
For any $1\not\equiv\chi\in\widehat{\T}^4$ whenever $r,s\in\N$ are odd, relatively prime and large enough then $\int_I \chi \ d \lambda_{I}=0$ for each ergodic component of $(T^r\times T^s)_{\Psi}$, where $\lambda_I$ is the relevant invariant measure.
\end{pr}
\begin{proof}
The assertion follows immediately from Proposition~\ref{pr:11}, Lemma~\ref{lm:12} and Lemma~\ref{lm:13}.
\end{proof}

\begin{proof}[Proof of Theorem~\ref{thm:affine} in case $c=1$, $\alpha\not\in\Q$]\label{totu}
We only need to check that \eqref{numer:a}, \eqref{numer:b}, \eqref{numer:c} hold. Conditions~\eqref{numer:a} and~\eqref{numer:b} follows from Proposition~\ref{pr:11}. It follows by Proposition~\ref{pr:14} that also condition~\eqref{numer:c} is satisfied (we take $\mathcal{F}=\widehat{\T}^2$).
\end{proof}

\subsubsection{Case  $c=0$, $\alpha \not\in\Q$}
\begin{proof}[Proof of Theorem~\ref{thm:affine} in case $c=0$, $\alpha\not\in\Q$]
We have $T_\psi(x,u)=(x+\alpha,u+\gamma)$, i.e.\ $T_\psi$ is a rotation on $\T^2$.  The decomposition into minimal components of $(T^r\times T^s)_\Psi$ consists of the cosets of $I_0=I_0^{r,s}=\{(x,y,0,0)\in\T^4\colon sx=ry\}$. They are at the same the ergodic components which are moreover uniquely ergodic, and we conclude as previously.
\end{proof}

\subsection{Generic case -- compact group extensions}
Let $f\colon \R\to\R$ (periodic of period $1$) be in $L^2(\T)$. Denote by
$$
f(x)=\sum_{n\in\Z}\widehat{f}(n) e^{2\pi i nx}
$$
the Fourier expansion of $f$. Recall that our goal is to prove disjointness of
$$
\T^2\ni (x,y)\mapsto (x+\alpha,y+cx+f(x))\in\T^2,
$$
with the M\"obius function $\bmu$ for a generic set of $\alpha$ (under some additional assumptions on $f$).

Recall the following result.
\begin{tw}[\cite{katokbook2003},~\cite{MR1286834}\footnote{In~\cite{katokbook2003} $f$ is assumed to be of class $C^2$ and the proof in fact requires that $\sum n|\widehat{f}(n)|<\infty$. See~\cite{MR1286834} for the proof of Theorem~\ref{thm:katok} with $f\in C^{1+\delta}(\T)$ for some $\delta>0$.}]\label{thm:katok}

Suppose that $f(x)=\sum_{n\in\Z}\widehat{f}(n) e^{2\pi i nx}$ is in $C^{1+\delta}(\T)$ for some $\delta>0$, and has zero mean. Denote by $T\colon \T\to \T$ the irrational rotation $x\mapsto x+\alpha$. Assume that for a sequence $(p_n/q_n)_{n\in\N}$ of rational numbers we have
\begin{equation}\label{eq:28}
\frac{|\widehat{f}({q_n})|}{\sum_{k\geq 1}|\widehat{f}({kq_n})|}>c>0
\end{equation}
and
\begin{equation}\label{eq:29}
\frac{\left| \alpha-\frac{p_n}{q_n}\right|q_n}{|\widehat{f}({q_n})|}\to 0.
\end{equation}
Then for each $\lambda\in\mathbb{S}^1$ the cocycle $\lambda e^{2\pi i f (\cdot)}$ is not a $T$-coboundary.
\end{tw}

We will now prove a modified version of the above theorem. It will be our main tool in course of the proof of Theorem~\ref{thm:main}.
\begin{tw}\label{tw:25}
Suppose that $f(x)=\sum_{n\in\Z}\widehat{f}(n) e^{2\pi i nx}$ is in $C^{1+\delta}(\T)$ for some $\delta>0$, and has zero mean. Denote by $T\colon \T\to \T$ the irrational rotation $x\mapsto x+\alpha$. Let $r,s\in\N$ ($r>s$) be relatively prime. Assume that $(p_n/q_n)_{n\in\N}$ is a subsequence of convergents of $\alpha$ in its continued fraction expansion such that
\begin{equation}\label{eq:28a}
\frac{|\widehat{f}(q_n)|}{\sum_{k\geq 1}|\widehat{f}(kq_n)|}>c_0>0,
\end{equation}
\begin{equation}\label{eq:28b}
\frac{|\widehat{f}(q_n)|}{|\widehat{f}(q_n)|+\sum_{k\geq 1}|\widehat{f}({k\frac{r}{s}q_n})|}>c_0>0,\text{ whenever }s|q_n,
\end{equation}
and
\begin{equation}\label{eq:29a}
\frac{\left| \alpha-\frac{p_n}{q_n}\right|q_n}{|\widehat{f}(q_n)|}\to 0.
\end{equation}
Then for each $\lambda\in\mathbb{S}^1$, $h\in\R$ and $A,B\in \R$ with $A^2+B^2\neq 0$ the cocycle $\lambda e^{2\pi i (Af^{(r)}(rx)+Bf^{(s)}(sx+h))}$ is not a $T$-coboundary.
\end{tw}
\begin{proof}
Fix $A,B\in\R$ with $A^2+B^2\neq 0$, relatively prime numbers $r>s$ and $h\in\R$. Set
$$
F(x):=Af^{(r)}(rx)+Bf^{(s)}(sx+h).
$$
For any $k\in \N$ we have
$$
f^{(k)}(x)=\sum_{n\in\Z}\widehat{f}(n)\frac{1-e^{2\pi i n k \alpha}}{1-e^{2\pi i n \alpha}}e^{2\pi i n x}.
$$
Therefore,
\begin{equation}\label{eq:ftilde}
F(x)
=A\sum_{n\in \Z} \widehat{f}(n) \frac{1-e^{2\pi i n r \alpha}}{1-e^{2\pi i n \alpha}}e^{2\pi i nr x}+B\sum_{n\in \Z}\widehat{f}(n)\frac{1-e^{2\pi i n s \alpha}}{1-e^{2\pi i n\alpha}}e^{2\pi i n (sx+h)}.
\end{equation}

Suppose first that $A\cdot B=0$. We may assume without loss of generality that $A=0$ and $B\neq 0$. Then, for all $n\in \N$, by~\eqref{eq:ftilde}, we have
$$
\widehat{F}({sn})=B\widehat{f}(n)\frac{1-e^{2\pi i ns\alpha}}{1-e^{2\pi i n \alpha}}e^{2\pi i nh}.
$$
Therefore, since
\begin{equation}\label{EQ:graniczne}
\left|\frac{1-e^{2\pi i q_ns\alpha}}{1-e^{2\pi i q_n \alpha}}\right|\to s\text{ as }n\to \infty,
\end{equation}
\begin{equation}\label{EQ:1}
|\widehat{F}({sq_n})|\geq |B|\cdot\frac{s}{2}\cdot |\widehat{f}(q_n)| \text{ for $n$ sufficiently large}
\end{equation}
and
\begin{equation}\label{EQ:22}
|\widehat{F}(msq_n)|=|B|\cdot|\widehat{f}(mq_n)|\cdot \left|\frac{1-e^{2\pi i msq_n\alpha}}{1-e^{2\pi i m q_n\alpha}} \right|\leq |B|\cdot s\cdot |\widehat{f}(mq_n)|
\end{equation}
for all $n\in \N$, $m\in\Z$. It follows immediately by~\eqref{EQ:1} and~\eqref{EQ:22} that for $n$ sufficiently large
\begin{equation}\label{EQ:1a}
\frac{|\widehat{F}(sq_n)|}{\sum_{m\geq 1}|\widehat{F}(msq_n)|}\geq\frac{|B|\cdot\frac{s}{2}\cdot |\widehat{f}(q_n)|}{\sum_{m\geq 1}|B|\cdot s \cdot |\widehat{f}(mq_n)|}=\frac{|\widehat{f}(q_n)|}{2\sum_{m\geq 1}|\widehat{f}(mq_n)|}
\end{equation}
\begin{equation}\label{EQ:2a}
\frac{\left|\alpha-\frac{sp_n}{sq_n}\right|sq_n }{|\widehat{F}(sq_n)|}\leq \frac{2}{|B|}\cdot \frac{\left|\alpha-\frac{p_n}{q_n}\right|q_n }{|\widehat{f}(q_n)|}.
\end{equation}
Now,~\eqref{EQ:1a} and~\eqref{EQ:2a} and the assumptions~\eqref{eq:28a} and~\eqref{eq:29a} imply that for $n$ sufficiently large
$$
\frac{|\widehat{F}(sq_n)|}{\sum_{m\geq 1}|\widehat{F}(msq_n)|}>\frac{c_0}{2}>0\text{ and }\frac{\left|\alpha-\frac{sp_n}{sq_n}\right|sq_n }{|\widehat{F}(sq_n)|}\to 0.
$$
In view of Theorem~\ref{thm:katok}, this completes the proof in case $A\cdot B=0$.

Suppose now that $A\cdot B\neq 0$. Applying the fact that for any absolutely summable sequence $(y_n)_{n\in\Z}$
$$
\sum_{n\in\Z}y_n=\sum_{r|n}y_n+\sum_{r\nmid n}y_n=\sum_{n\in\Z}y_{nr}+\sum_{r\nmid n}y_n=\sum_{s|n}y_{n\frac{r}{s}}+\sum_{r\nmid n}y_n\footnote{These equalities hold for all $r,s\in\N$.}
$$
to the second summand in formula~\eqref{eq:ftilde}, we obtain
\begin{multline}\label{eq:mm}
F(x)=\sum_{s| n}\underbrace{\left(A \widehat{f}(n)\frac{1-e^{2\pi i nr \alpha}}{1-e^{2\pi i n\alpha}}+B \widehat{f}(n\frac{r}{s})\frac{1-e^{2\pi i nr \alpha}}{1-e^{2\pi i n\frac{r}{s}\alpha}}e^{2\pi i n\frac{r}{s}h} \right)}_{a_{rn}}e^{2\pi i nrx}\\
+A\sum_{s\nmid n}\widehat{f}(n)\frac{1-e^{2\pi i rn \alpha}}{1-e^{2\pi i n \alpha}}e^{2\pi i n rx}+B\sum_{r\nmid n}\widehat{f}(n) \frac{1-e^{2\pi i ns \alpha}}{1-e^{2\pi i n \alpha}}e^{2\pi i n h}e^{2\pi i ns x}.
\end{multline}
For $m\in \N$, let $B_m:=\{n\in\N \colon m\nmid q_n\}$. We will consider the following cases:
\begin{enumerate}[(i)]
\item\label{A1}
$B_r$ or $B_s$ is infinite,
\item\label{A2}
both sets $B_r$ and $B_s$ are finite.
\end{enumerate}
We will cover first case~\eqref{A1}. Without loss of generality we may assume that $B_r$ is infinite.
Suppose that
\begin{equation}\label{eq:n1}
\lambda e^{2\pi i (Af^{(r)}(rx)+Bf^{(s)}(sx+h))} \text{ is a coboundary}.
\end{equation}
Since $f^{(r)}(rx)$ is $\frac{1}{r}$-periodic, it follows immediately that
\begin{equation}\label{eq:n2}
\lambda e^{2\pi i (Af^{(r)}(rx)+Bf^{(s)}(sx+\frac{s}{r}+h))}\text{ is also a coboundary. }
\end{equation}
Hence, dividing the expression from~\eqref{eq:n1} by the one from~\eqref{eq:n2}, we conclude that
\begin{equation}\label{eq:kobrzeg}
e^{2\pi i B(f^{(s)}(sx)-f^{(s)}(sx+\frac{s}{r}))}\text{ is a coboundary as well.}
\end{equation}
We will show that this is impossible in view of Theorem~\ref{thm:katok}. Indeed, we have
$$
g(x):=f^{(s)}(sx)-f^{(s)}\left(sx+\frac{s}{r}\right)=\sum_{n\in\Z}\widehat{f}(n)\frac{1-e^{2\pi i ns \alpha}}{1-e^{2\pi i n \alpha}}(1-e^{2\pi i n\frac{s}{r}})e^{2\pi i n s x}.
$$
We claim that for $n\in B_r$,
\begin{equation}\label{eq:9}
\frac{|\widehat{g}(sq_n)|}{\sum_{k\geq 1}|\widehat{g}(ksq_n)|}>c_1 \text{ for some }c_1>0
\end{equation}
and
\begin{equation}\label{eq:10}
\frac{\left| \alpha-\frac{sp_n}{sq_n}\right|sq_n}{|\widehat{g}(sq_n)|}\to 0.
\end{equation}
Indeed, for $n\in B_r$, we have
\begin{equation}\label{EQ:2}
|\widehat{g}(sq_n)|=|\widehat{f}(q_n)|\cdot \left|\frac{1-e^{2\pi i s q_n\alpha}}{1-e^{2\pi i q_n \alpha}} \right|\cdot |e^{2\pi i q_n\frac{s}{r}}-1|.
\end{equation}
Since $r\nmid q_n$ and $r$ and $s$ are relatively prime,
\begin{equation}\label{EQ:3}
\left|e^{2\pi i q_n\frac{s}{r}}-1 \right|\geq \left|e^{2\pi i \frac{1}{r}}-1 \right| =:c_2>0.
\end{equation}
Using again~\eqref{EQ:graniczne}, we obtain from~\eqref{EQ:2} and~\eqref{EQ:3} that for $n$ sufficiently large,
\begin{equation}\label{eq:a}
|\widehat{g}(sq_n)|\geq \frac{c_2s}{2} \cdot|\widehat{f}(q_n)|.
\end{equation}
On the other hand, since $\left|\frac{1-e^{2\pi i k s q_n \alpha}}{1-e^{2\pi i k q_n \alpha}} \right|\leq s$,
\begin{align}
\begin{split}\label{eq:b}
\sum_{k\geq 1}|\widehat{g}(ksq_n)|&=\sum_{k\geq 1}|\widehat{f}(kq_n)|\cdot\left|\frac{1-e^{2\pi i k s q_n \alpha}}{1-e^{2\pi i k q_n \alpha}} \right| \cdot |e^{2\pi i q_n\frac{s}{r}}-1|\\
&\leq \sum_{k\geq 1}|\widehat{f}(kq_n)|\cdot s \cdot |e^{2\pi i q_n\frac{s}{r}}-1|\leq 2s\sum_{k\geq 1}|\widehat{f}(kq_n)|.
\end{split}
\end{align}
Thus, using~\eqref{eq:a},~\eqref{eq:b} and~\eqref{eq:28a} we obtain, for $n$ large enough,
$$
\frac{|\widehat{g}(sq_n)|}{\sum_{k\geq 1}|\widehat{g}(ksq_n)|}\geq \frac{c_2 s |\widehat{f}(q_n)|}{4s\sum_{k\geq 1}|\widehat{f}(kq_n)|}>\frac{c_0c_2}{4},
$$
which shows that~\eqref{eq:9} holds. Using~\eqref{eq:a} and~\eqref{eq:29a}, we obtain that~\eqref{eq:10} also holds:
$$
\frac{\left| \alpha-\frac{sp_n}{sq_n}\right|sq_n}{|\widehat{g}(sq_n)|}\leq \frac{2\left| \alpha-\frac{p_n}{q_n}\right|sq_n}{c_2 s|\widehat{f}(q_n)|}=\frac{2}{c_2}\cdot\frac{\left| \alpha-\frac{p_n}{q_n}\right|q_n}{|\widehat{f}({q_n})|}\to 0.
$$
It follows from Theorem~\ref{thm:katok} that $\lambda e^{2\pi i B(f^{(s)}(sx)-f^{(s)}(sx+\frac{s}{r}))}$ cannot be a coboundary, contrary to~\eqref{eq:kobrzeg}. This completes the proof in case (i).

We cover now case~\eqref{A2}. We claim that for $n\not\in B_s$ (i.e.\ $n$ such that $s|q_n$), for some $c_3>0$, we have
\begin{equation}\label{28d}
\frac{|a_{rq_n}|}{\sum_{m\geq 1}|a_{mrq_n}|}>c_3>0
\end{equation}
and
\begin{equation}\label{29d}
\frac{|\alpha-\frac{rp_n}{rq_n}|rq_n}{|a_{rq_n}|}\to 0.
\end{equation}
To prove~\eqref{28d}, we will estimate $|a_{rq_n}|$ from below and $\sum_{m\geq 1}|a_{mrq_n}|$ from above in an appropriate way.
We begin by estimating $|a_{rq_n}|$. We have
\begin{multline}\label{eq:kropa1}
|a_{rq_n}|\geq |A|\cdot |\widehat{f}(q_n)|\cdot \left|\frac{1-e^{2\pi i rq_n\alpha}}{1-e^{2\pi i q_n\alpha}} \right|-|B|\cdot \left|\widehat{f}\left(q_n\frac{r}{s}\right)\right|\cdot \left|\frac{1-e^{2\pi i q_nr\alpha}}{1-e^{2\pi i q_n\frac{r}{s}}\alpha} \right|\\
\geq |A|\cdot |\widehat{f}(q_n)|\cdot \frac{r}{2}-|B|\cdot\left|\widehat{f}\left(q_n\frac{r}{s}\right)\right|\cdot\left|\frac{1-e^{2\pi i q_nr\alpha}}{1-e^{2\pi i q_n\frac{r}{s}\alpha}} \right|,
\end{multline}
where the latter inequality follows from~\eqref{EQ:graniczne} and is valid for $n$ sufficiently large.
It follows by~\eqref{eq:28b} that for $n\not\in B_s$,
\begin{equation}\label{eq:kropa2}
\left|\widehat{f}\left(q_n\frac{r}{s}\right)\right|<\frac{1}{c_0}|\widehat{f}(q_n)|.
\end{equation}
We will now estimate $|1-e^{2\pi i q_n\frac{r}{s}\alpha} |$. Notice that
\begin{equation}\label{EQ:4}
\frac{2}{\pi}<\frac{|e^{2\pi i x}-e^{2\pi i y}|}{|x-y|}<1
\end{equation}
for $x,y\in\R$ such that $0<|x-y|<1$. Since $|1-e^{2\pi i q_n\alpha}|\to 0$, for $n$ sufficiently large,
\begin{equation*}\label{EQ:5}
\left|1-e^{2\pi i q_n\alpha} \right|<\frac{s}{\pi r}\left|1-e^{2\pi i \frac{1}{s}} \right|.
\end{equation*}
Therefore and by~\eqref{EQ:4}, for such $n$, we have
\begin{align}
\begin{split}\label{eq:nn}
\left|e^{2\pi i \frac{[q_n\alpha]}{s}}-e^{2\pi i q_n\frac{1}{s}\alpha}\right|&< \left|\frac{[q_n\alpha]}{s}-q_n\frac{1}{s}\alpha\right|=\frac{1}{s}\cdot |[q_n\alpha]-q_n\alpha|\\
&<\frac{\pi}{2s}\cdot |e^{2\pi i [q_n\alpha]}-e^{2\pi i q_n\alpha}|=\frac{\pi}{2s}\cdot |1-e^{2\pi i q_n\alpha}|\\
&<\frac{\pi}{2s}\cdot\frac{s}{\pi r}\cdot |1-e^{2\pi i \frac{1}{s}}|=\frac{1}{2r}\cdot \left|1-e^{2\pi i \frac{1}{s}}\right|.
\end{split}
\end{align}
Since for all $x,y\in\R$ and all $k\in\N$
$$
|e^{2\pi i k x}-e^{2\pi i ky}|\leq k|e^{2\pi i x}-e^{2\pi i y}|,
$$
it follows from~\eqref{eq:nn} that
$$
|e^{2\pi i \frac{[q_n\alpha]r}{s}}-e^{2\pi i q_n\frac{r}{s}}\alpha|<\frac{1}{2}|1-e^{2\pi i \frac{1}{s}}|.
$$
Since $q_n$ is a denominator of $\alpha$ and $s|q_n$, we have $|c-\frac{q_n}{s}\alpha|>|p_n-q_n\alpha|$ for all $0\leq c\leq \frac{q_n}{s}$ (see\eqref{eq:best} in Section~\ref{se:rot}). Using this inequality and~\eqref{EQ:4}, we obtain
\begin{equation}\label{EQ:6}
\frac{\pi}{2}|1-e^{2\pi i \frac{q_n}{s}\alpha}|>|c-\frac{q_n}{s}\alpha|>|p_n-q_n\alpha|>|1-e^{2\pi i q_n\alpha}|.
\end{equation}
By the first two lines of~\eqref{eq:nn} and~\eqref{EQ:6}, we have
\begin{equation}\label{EQ:7}
\left|e^{2\pi i \frac{[q_n\alpha]}{s}}-e^{2\pi i q_n\frac{1}{s}\alpha} \right|<\frac{\pi}{2s}|1-e^{2\pi i q_n\alpha}|<\frac{\pi}{2s}\cdot\frac{\pi}{2}|1-e^{2\pi i \frac{q_n}{s}\alpha}|.
\end{equation}
Suppose that $s|[q_n\alpha]$. Then~\eqref{EQ:7} implies $1<\frac{\pi}{2s}\cdot\frac{\pi}{2}$, i.e. $s<(\pi/2)^2$. This is however impossible, as $s\geq 3$. Therefore $s\nmid [q_n\alpha]$, which implies $s\nmid [q_n\alpha]r$, i.e.
$$
\left|1-e^{2\pi i \frac{[q_n\alpha]r}{s}} \right|\geq |1-e^{2\pi i \frac{1}{s}}|.
$$
Therefore
\begin{multline*}
\left|1-e^{2\pi i q_n\frac{r}{s}\alpha} \right|\geq \left|1-e^{2\pi i \frac{[q_n\alpha]r}{s}}\right|-\left|e^{2\pi i \frac{[q_n\alpha]r}{s}}-e^{2\pi i q_n\frac{r}{s}\alpha}\right|\\
\geq \left|1-e^{2\pi i \frac{[q_n\alpha]r}{s}} \right|-\frac{1}{2}\left|1-e^{2\pi i \frac{1}{s}} \right|\geq \frac{1}{2}\left|1-e^{2\pi i \frac{1}{s}} \right|=:c_4>0.
\end{multline*}
It follows that
$$
\left|\frac{1-e^{2\pi i rq_n\alpha}}{1-e^{2\pi i q_n\frac{r}{s}\alpha}} \right|\leq \frac{1}{c_4}|1-e^{2\pi i rq_n\alpha} |\leq \frac{r}{c_4}|1-e^{2\pi i q_n\alpha}|\to 0.
$$
This and~\eqref{eq:kropa1},~\eqref{eq:kropa2} imply that for $n$ sufficiently large
\begin{equation}\label{eq:nju}
|a_{rq_n}|\geq c_5 |\widehat{f}(q_n)|\text{ for some }c_5>0.
\end{equation}

We will estimate now $\sum_{m\geq 1}|a_{mrq_n}|$. By~\eqref{eq:28a} and~\eqref{eq:28b} we have
$$
\sum_{m\geq 1}|\widehat{f}(mq_n)|<\frac{1}{c_0}|\widehat{f}(q_n)|\text{ and }\sum_{m\geq 1}\left|\widehat{f}\left(mq_n\frac{r}{s}\right)\right|<\frac{1}{c_0}|\widehat{f}(q_n)|.
$$
Hence
$$
\sum_{m\geq 1}|a_{mrq_n}|\leq \frac{|A|r+|B|s}{c_0}\cdot |\widehat{f}(q_n)|.
$$
Using this estimate and~\eqref{eq:nju}, we obtain
$$
\frac{|a_{rq_n}|}{\sum_{m\geq 1}|a_{mrq_n}|}\geq\frac{c_5|\widehat{f}(q_n)|}{\frac{|A|r+|B|s}{c_0}|\widehat{f}(q_n)|}=\frac{c_0c_5}{|A|r+|B|s}>0
$$
and~\eqref{28d} follows. Notice that~\eqref{eq:nju} together with~\eqref{eq:29a} implies that also~\eqref{29d} is true. By Theorem~\ref{thm:katok}, we conclude that $\lambda e^{2\pi i F(x)}$ cannot be a coboundary, which completes the proof in case~\eqref{A2}.
\end{proof}

\begin{uw}\label{uw:27}
Recall that if $f\in C^{1+\delta}(\T)$ for some $\delta>0$ then $\widehat{f}(n)={\rm o}(1/n^{1+\delta'})$ for $0<\delta'<\delta$ (the speed of convergence to zero depends on $\delta'$).
\end{uw}

\begin{lm}[cf. Lemma 4 in~\cite{MR1286834}]\label{lm:4a}
Let $g\colon \N\to (0,\infty)$ be a non-increasing positive function such that $g(mn)\leq g(m)g(n)$ for all $m,n\in\N$ and $\sum_{m\geq 1} g(m)=C<\infty$.  Let $(x_n)_{n\in\N}\subset [0,\infty)$ be a summable sequence such that $x_n=o(g(n))$, $n\in\N$. Let $(x_{n_k})_{k\in\N}$ be a subsequence of $(x_n)_{n\in\N}$ such that $x_{n_k}>0$. Let $b\geq 1$ and let
$$
\vep_k=\frac{x_{n_k}}{x_{n_k}+\sum_{m\geq 1}x_{m[bn_k]}}.
$$
Then $\vep_k\not\to 0$.
\end{lm}
\begin{proof}
We will choose a subsequence of $(\vep_k)_{k\in \N}$ recursively.
Let $k_1\geq 1$ and $\delta_1>0$ be such that
\begin{equation*}
\frac{x_{n_{k_1}}}{g(n_{k_1})}>\delta_1 \text{ and }\frac{x_n}{g(n)}\leq \delta_1 \text{ for }n>n_{k_1}.
\end{equation*}
Suppose first that $[bn_{k_1}]=n_{k_1}$. Then
\begin{multline*}
x_{n_{k_1}}+\sum_{m\geq 1}x_{m[bn_{k_1}]}=2x_{n_{k_1}}+\sum_{m\geq 2}x_{m[bn_{k_1}]}\leq 2x_{n_{k_1}}+\sum_{m\geq 2}g(m[bn_{k_1}])\delta_1\\
\leq 2x_{n_{k_1}}+g([bn_{k_1}])\delta_1\sum_{m\geq 2}g(m)\leq 2x_{n_{k_1}}+g(n_{k_1})\delta_1C<(C+2)x_{n_{k_1}}.
\end{multline*}
Suppose now that $[bn_{k_1}]>n_{k_1}$. In a similar way as before, we obtain
\begin{equation*}
x_{n_{k_1}}+\sum_{m\geq 1}x_{m[bn_{k_1}]}\leq x_{n_{k_1}}+g(n_{k_1})\delta_1\sum_{m\geq 1}g(m)<(C+1)x_{n_{k_1}}.
\end{equation*}

Once we have chosen $k_1,\dots, k_j$, we pick $k_{j+1}>k_j$ and $\delta_{j+1}>0$ such that
$$
\frac{x_{n_{k_{j+1}}}}{g(n_{k_{j+1}})}>\delta_{j+1} \text{ and }\frac{x_n}{g(n)}\leq \delta_{j+1} \text{ for }n>n_{k_{j+1}}.
$$
As before, we obtain
$$
\sum_{m\geq 1}x_{m[bn_{k_{j+1}}]}<(C+2) x_{n_{k_{j+1}}}
$$
which completes the proof.
\end{proof}

\begin{uw}\label{uw:29}
Given $b_1,\dots,b_l\geq1$, under the assumptions of the above lemma, by a diagonalizing procedure we can find an increasing sequence $(n_k)\subset \N$ such that
$$
\vep_{n_k}(b_i) >c >0\text{ for $1\leq i\leq l$}.
$$
\end{uw}

\begin{wn}\label{wn:21-}
Suppose that $f\in C^{1+\delta}(\T)$ for some $\delta>0$ and it is not a trigonometric polynomial. Then for a generic $\alpha$ and any $A,B\in \R^2$ with $A^2+B^2\neq 0$, any relatively prime numbers $r$ and $s$, any $h\in \R$ and any $\lambda\in\T$, the cocycle $\lambda e^{2\pi i (Af^{(r)}(rx)+Bf^{(s)}(sx+h))}$ is not a $T$-coboundary for $Tx=x+\alpha$.
\end{wn}
\begin{proof}
Let $(q_n)_{n\in\N}$ be such that $\widehat{f}(q_n)\neq 0$. Then by~\eqref{eq:rezydualny}, for a residual set of irrationals $\alpha$, we have
$$
\frac{\left|\alpha-\frac{p_n}{q_n}\right|q_n}{|\widehat{f}(q_n)|}\to 0
$$
along some subsequence of $(q_n)_{n\in\N}$ which, for convenience, we will still denote by $(q_n)_{n\in\N}$. In other words,~\eqref{eq:29a} holds. It follows from Remark~\ref{uw:27}, Lemma~\ref{lm:4a} and Remark~\ref{uw:29} (for $l=2$, $b_1=1$, $b_2=\frac{r}{s}$) that~\eqref{eq:28a} and~\eqref{eq:28b} also hold. Therefore, we can apply Theorem~\ref{tw:25} to complete the proof.
\end{proof}

\begin{uw}\label{uw:stop}
Let $c\neq 0$, $r,s,A,B\in\Z$ be such that $Ar^2+Bs^2\neq 0$. Let $\varphi(x)=cx+f(x)$, where $f\colon \R\to \R$ is of class $C^{1+\delta}$ for some $\delta>0$ and periodic of period $1$. Since the topological degree of $e^{2\pi i (A\varphi^{(r)}(r\cdot)+B\varphi^{(s)}(s\cdot + h))}$ is equal to $(Ar^2+Bs^2)c$, an immediate consequence of~\eqref{eq:fure} is that the cocycle
$e^{2\pi i (A\varphi^{(r)}(r\cdot)+B\varphi^{(s)}(s\cdot + h))}$ is not a coboundary for all $h\in\R$. On the other hand, if $Ar^2+Bs^2=0$, then for some $\lambda$ of modulus~1, we have
$$
e^{2\pi i (A\varphi^{(r)}(r\cdot)+B\varphi^{(s)}(s\cdot + h))}=\lambda\cdot
e^{2\pi i (Af^{(r)}(r\cdot)+Bf^{(s)}(s\cdot + h))}.$$
\end{uw}

\begin{wn}\label{wn:21}
Let $c\in\Z$ and suppose that $f\in C^{1+\delta}(\T)$ for some $\delta>0$ and it is not a trigonometric polynomial. Let $\varphi(x)=cx+f(x)$. Then for a generic $\alpha$ the cocycle $(e^{2\pi i \varphi^{(r)}(r\cdot)}, e^{2\pi i \varphi^{(s)}(s\cdot +h)})$ is ergodic (as a cocycle over $Tx=x+\alpha$) for all relatively prime numbers $r\neq s$ and any $h\in\R$.
\end{wn}
\begin{proof}
In view of Remark~\ref{uw:7}, the assertion follows immediately from Corollary~\ref{wn:21-} and Remark~\ref{uw:stop}.
\end{proof}

\begin{proof}[Proof of Theorem~\ref{thm:main}]
Recall that $T_{c,f}(x,y)=(x+\alpha,y+cx+f(x))$. We can assume that $f$ is not a trigonometric polynomial. Indeed, otherwise $f$ is a coboundary with the transfer function also being a trigonometric polynomial and the problem is reduced to the affine case.

It follows from Lemma~\ref{lm:7} and Corollary~\ref{wn:21} that for a generic $\alpha$ the decomposition of $\T^4$ into minimal components of $({T_{c,f}})^r\times ({T_{c,f}})^s$ is the same as the decomposition into ergodic components: it consists of sets $I_{c_1}$, $c_1\in [0,\frac{1}{r})$ (see~\eqref{eq:i} on page~\pageref{eq:i}). Moreover, each ergodic component is uniquely ergodic. It follows that all points are generic for $({T_{c,f}})^r\times ({T_{c,f}})^s$ (for the relevant invariant measures). Thus, conditions~\eqref{numer:a} and~\eqref{numer:b} are satisfied.

To complete the proof, it suffices to show that the set $\widehat{\T}^2$ satisfies condition~\eqref{numer:c}. This is however true by Lemma~\ref{lm:12} and the result follows.
\end{proof}

\section{Appendix}

\subsection{Notation and basic facts}

\paragraph{Automorphisms of standard Borel spaces}
Let $(X,\mathcal{B},\mu)$ be a  probability standard Borel space. By ${\rm Aut}(X,\mathcal{B},\mu)$ we will denote the space of all bi-measurable measure-preserving bijections of $X$ which we will call automorphisms.

\paragraph{Irrational rotation}\label{se:rot}
We will identify the multiplicative circle $\mathbb{S}^1=\{z\in\C\colon |z|=1\}$ and $\T=\R/\Z$ with $X=[0,1)$ with addition mod $1$. Therefore, real functions defined on the circle will be identified with one-periodic functions defined on $\R$. Let $\lambda_\T$ denote Lebesgue measure on $X$.

Assume that $T\colon X\to X$ is an irrational rotation, $Tx=x+\alpha\ (\text{mod } 1)$, $x\in X$. Clearly $T\in {\rm Aut}(X,\mathcal{B}(X),\lambda_\mathbb{T})$. Let $\alpha=[0;a_1,a_2,\dots]$ be the continued fraction expansion of $\alpha$. Let
\begin{align*}
&q_0=1, q_1=a_1, q_{n+1}=a_{n+1}q_n+q_{n-1},\\
&p_0=0, p_1=1, p_{n+1}=a_{n+1}p_n+p_{n-1}
\end{align*}
for $n\geq 1$.
The rationals $p_n/q_n$ are called the \emph{convergents} of $\alpha$.

Recall (see e.g.~\cite{MR1451873}) that every convergent $p_n/q_n$ is a best approximation of $\alpha$ in the following sense:
\begin{equation}\label{eq:best}
\text{ if }\frac{c}{d}\neq\frac{p_n}{q_n} \text{ and }0<d\leq q_n \text{ then }|c-d\alpha|>|p_n-q_n\alpha|.
\end{equation}

Recall also (see e.g.\ \cite{MR1286834}) that given an infinite set $\{q_n\}_{n\in\N}\subset \N$ and a positive real valued function $R=R(q_n)$ the set
\begin{multline}\label{eq:rezydualny}
\mathcal{A}=\left\{\alpha\in [0,1) \colon \text{ for infinitely many }n\text{ we have }\left|\alpha-\frac{p_n}{q_n} \right|<R(q_n),\right.\\
\left.\text{where }\frac{p_n}{q_n}\text{ are convergents of }\alpha\right\}\text{ is residual in }\T.
\end{multline}

\paragraph{Cocycles and group extensions}
Let $T\in {\rm Aut}(X,\mathcal{B},\mu)$. For a locally compact second countable Abelian group $G$\footnote{We use multiplicative notation in $G$.} with a Haar measure $\lambda_G$, a measurable map $\varphi\colon \Z\times X\to G$ is called a \emph{cocycle} if
$$
\varphi^{(n+m)}(x)={\varphi}^{(n)}(x){\varphi}^{(m)}(T^nx)\text{ for each }n,m\in\Z.
$$
The generator $\varphi(x)={\varphi}^{(1)}(x)$ determines ${\varphi}^{(n)}(x)$ for any $(n,x)\in \Z\times X$:
$$
\varphi^{(n)}(x)=\begin{cases}
\varphi(x)\cdot\varphi(Tx)\cdot\ldots \cdot \varphi(T^{n-1}x),& \text{ if }n>0,\\
1,& \text{ if }n=0,\\
(\varphi(T^nx)\cdot\ldots\cdot\varphi(T^{-1}x))^{-1},&\text{ if }n<0.
\end{cases}
$$
Thus, we will call a \emph{cocycle} any measurable function $\varphi\colon X\to G$ as well. Given $\varphi$, we consider a $G$-extension of $T$, acting on $(X\times G,\mu\otimes \lambda_G)$, defined by the formula
$$
T_\varphi(x,g)=(Tx,\varphi(x)g).
$$
A cocycle $\varphi$ is called a $T$-\emph{coboundary} (or simply a \emph{coboundary}) if it is of the form
$$
\varphi(x)=\xi(x)(\xi(Tx))^{-1}
$$
for some measurable function $\xi\colon X\to G$ (called a \emph{transfer function}). Two cocycles $\phi,\psi\colon X\to G$ are \emph{cohomologous} if for some measurable function $f\colon X\to G$ we have
\begin{equation}\label{eq:gwizd}
(f(Tx))^{-1}\phi(x)f(x)=\psi(x).
\end{equation}

Analogous notions to the above ones are also present in topological dynamics. Let $T\colon X\to X$ be a minimal homeomorphism of a compact metric space. Let $\varphi\colon X\to G$ be a continuous function. We say that $\varphi$ is a \emph{topological $T$-coboundary} if it is a measurable $T$-coboundary with a continuous transfer function.

\subsection{Compact group extensions: ergodicity and minimality}

Let $G$ be a compact Abelian metrizable group and assume that $T\in {\rm Aut}(X,\mathcal{B},\mu)$ is ergodic. Let $\varphi:X\to G$ be a cocycle.
Then
\begin{equation}\label{cerg}
\begin{array}{l}
\mbox{$T_\varphi$ is ergodic if and only if the equation}\\
\mbox{$\chi\circ\varphi=\xi/\xi\circ T$ has only trivial solution $\chi=\raz, \xi={\rm  const}$}\\
\mbox{in $\chi\in\widehat{G}$ and a measurable $\xi:X\to\mathbb{S}^1$.}\end{array}\end{equation}
For example, see \cite{MR0133429},
\begin{equation}\label{eq:fure}\begin{array}{l}
\mbox{if $Tx=x+\alpha$, $\varphi:\T\to\mathbb{S}^1$ is Lipschitz with non-zero degree}\\
\mbox{then $T_\varphi$ is ergodic.}\end{array}\end{equation}
If $T_\varphi$ is ergodic, we will also say that $\varphi$ is ergodic (for $T$).

Recall that $G$ acts on $X\times G$ via $g\mapsto R_g$, where:
$$
R_g(x,h):=(x,hg^{-1}) \text{ for }(x,h)\in X\times G,\ g\in G.
$$

\paragraph{Ergodic components}
Let $\varphi\colon X\to G$ be a cocycle. Let $\mathcal{P}(T_\varphi,\mathcal{B},\mu)$ stand for the set of $T_\varphi$-invariant Borel measures whose projection on $X$ is $\mu$. Fix an ergodic measure $\lambda\in \mathcal{P}(T_\varphi,\mathcal{B},\mu)$. Denote by $H$ be the stabilizer of $\lambda$ in $G$, i.e.
$$
H:=\{g\in G\colon \lambda\circ R_g=\lambda\}.
$$
Notice that
$$
\Gamma:=\{\chi\in\widehat{G}\colon \chi\circ \varphi\text{ is a coboundary}\}
$$
is a (closed) subgroup of $\widehat{G}$ and let
$$
F=\text{ann }\Gamma:=\{g\in G\colon \text{ for each }\chi\in\Gamma,\ \chi(g)=1 \}.
$$

\begin{pr}[see e.g.~\cite{MR1091425}]\label{pr:5}
The system $(X\times G,T_\varphi,\lambda)$ is isomorphic to $(X\times H,T_\psi,\mu\otimes \lambda_H)$ for some ergodic $\psi\colon X\to H$. Moreover, $F=H$.
\end{pr}
\begin{uw}\label{uw:7}
In view of~\eqref{cerg}, $T_\varphi$ is ergodic if and only if $\Gamma=\{\raz\}$.
\end{uw}

\paragraph{Minimal components}
Let $T\colon X\to X$ be a minimal homeomorphism of a compact metric space. Let $\varphi\colon X\to G$ be a continuous function and let $M$ be a minimal component of $T_\varphi$, i.e. $M\subset X\times G$ is closed and invariant with no proper subsets having the same properties. Let $H_{top}$ be the stabilizer of $M$ in $G$, i.e.
$$
H_{top}:=\{g\in G\colon  R_g(M)=M\}
$$
(this definition is independent of the initial choice of $M$). Notice that
$$
\Gamma_{top}:=\{\chi\in\widehat{G}\colon \chi\circ \varphi\text{ is a topological coboundary}\}
$$
is a closed subgroup of $\widehat{G}$ and let
$$
F_{top}:=\text{ann }\Gamma_{top}=\{g\in G\colon \text{ for each }\chi\in\Gamma_{top},\ \chi(g)=1 \}.
$$

\begin{lm}[\cite{phd-siemaszko}]\label{siem}
There exists a continuous map $\tau\colon X\to G/H_{top}$ such that
\begin{equation}\label{eq:siemaszko}
\tau(Tx)=\varphi(x)\tau(x).
\end{equation}
Moreover, $M:=\cup_{x\in X}\{x\}\times \tau(x)$ is a minimal set.
\end{lm}
The proof of the following result is analogous to the one in the measure-theoretical case. We include it here for the sake of completeness.
\begin{pr}\label{tw:top}
$H_{top}=F_{top}$.
\end{pr}
\begin{proof}
Let $g_0\in H_{top}$ and let $\chi\in \Gamma_{top}$, i.e.
$$
\chi\circ \varphi(x)=h(Tx)/h(x)
$$
for some continuous function $h\colon X\to \mathbb{S}^1$. We define $w\colon X\times G\to \mathbb{S}^1$:
$$
w(x,g)=h(x)^{-1}\chi(g).
$$
This function is continuous and $T_\varphi$-invariant, whence it is constant on each minimal component. It follows that
$h(x)^{-1}\chi(g)\chi(g_0)=w(x,gg_0)=w(x,g)=h(x)^{-1}\chi(g)$,
which implies $\chi(g_0)=1$. Therefore $g_0\in F_{top}$.

Suppose now that there exists $g_0\in F_{top}\setminus H_{top}$. Then there exists $\chi\in \widehat{G}$ such that $\chi(g_0)\neq 1$ and $\chi(H_{top})=\{1\}$. Let $\tau\colon X\to G/H_{top}$ satisfy~\eqref{eq:siemaszko}. It follows that $\chi \circ \tau$ is well-defined and we obtain
$$
\chi\circ \tau(Tx)=\chi\circ \varphi(x)\cdot \chi\circ \tau(x)
$$
which implies
$$
\chi\circ \varphi(x)=\frac{\chi\circ \tau(Tx)}{\chi\circ \tau (x)},
$$
i.e. $\chi\in \Gamma_{top}$ and consequently $\chi(g_0)=1$ which is a contradiction.
\end{proof}

\paragraph{Relation between the ergodic and the minimal components}

Let $T$ be a minimal homeomorphism of a compact metric space $X$ and let $\mu$ be a $T$-invariant probability Borel measure, ergodic with respect to $T$. Let $\varphi\colon X\to G$ be continuous. Then $T_\varphi$ is a homeomorphism of $X\times G$ and $T_\varphi\in {\rm Aut}(X\times G,\mathcal{B}\otimes \mathcal{B}(G),\mu\otimes \lambda_G)$. Let $\lambda\in\mathcal{P}(T_{\varphi},\mathcal{B},\mu)$. There are two natural partitions associated to $T_\varphi$:
\begin{itemize}
\item
$\mathcal{P}_{erg}$ -- the partition into the ergodic components of $T_\varphi$,
\item
$\mathcal{P}_{min}$ -- the partition into the minimal components of $T_\varphi$.
\end{itemize}
Partition $\mathcal{P}_{erg}$ is clearly measurable.
\begin{pr}
Partition $\mathcal{P}_{min}$ is measurable.
\end{pr}
\begin{proof}
Let $M\subset X \times G$ be as in Lemma~\ref{siem}, in particular, $M$ is a minimal component of $T_\varphi$. Let $s\colon G/H_{top}\to G$ be a Borel selector of the canonical projection $\pi\colon G\to G/H_{top}$. Consider $\eta\colon X\to G$ given by $\eta=s\circ\tau$. We obtain
$$
\varphi'(x):=\varphi(x)\eta(x)(\eta(Tx))^{-1} \in H_{top}
$$
and
$$
M=\bigcup_{x\in X}\{x\}\times(\eta(x)H_{top}),
$$
so the map
\begin{equation}\label{bbb}
(x,h)\mapsto (x,\eta(x)h)
\end{equation}
settles an equivariant Borel isomorphism of $X\times H$ (considered with $T_{\varphi'}$) and $M$ (considered with $T_\varphi$). Moreover,~\eqref{bbb} can be naturally extended to a Borel isomorphism of $X\times H_{top}\times G/H_{top}$ and $X\times G$. Clearly the partition of $X\times H_{top}\times G/H_{top}$ given by relevant translations of $X\times H_{top}\times \{1\}$ (indexed by $G/H_{top}$) is measurable for the product measure $\mu\otimes \lambda_{H_{top}}\otimes \lambda_{G/H_{top}}$. Hence its image by the Borel extensions of~\eqref{bbb} is also a measurable partition for the image of the measure $\mu\otimes \lambda_{H_{top}}\otimes \lambda_{G/H_{top}}$. This image is equal to $\mu\otimes \lambda_{G}$, so the partition into minimal components is indeed measurable.
\end{proof}
\begin{uw}\label{uuw:1}
Since the partition into the ergodic components can be defined as the finest measurable partition whose atoms are invariant under the action of the homeomorphism in question, $\mathcal{P}_{erg}$ is finer than $\mathcal{P}_{min}$.
\end{uw}
As a direct consequence of the above remark we obtain the following:
\begin{uw}\label{uw:9}
We have $H\supset H_{top}$. The condition $H=H_{top}$ is necessary and sufficient for the ergodic components of $T_\varphi$ to be the same as its minimal components. Moreover, $H=H_{top}$ if and only if $\Gamma=\Gamma_{top}$.
\end{uw}

\paragraph{Unique ergodicity}
\begin{pr}[Furstenberg, see the proof of Theorem I.4 in~\cite{MR0213508} and Proposition 3.10 in~\cite{MR603625}]
Let $T\colon X\to X$ be uniquely ergodic and let $\varphi\colon X\to G$ be a continuous cocycle with values in a compact Abelian group. If $T_\varphi$ is ergodic with respect to $\mu\otimes \lambda_G$ then it is uniquely ergodic.\label{pr:fur}
\end{pr}

\subsection{A remark on the KBSZ criterion for $Tx=x+\alpha$}
Let $Tx=x+\alpha$ be an irrational rotation on $\T$.
\begin{pr}\label{p1}
Let $A\subset \Z$.
\begin{enumerate}[(i)]
\item\label{H1}
If $A$ is such that
\begin{equation}\label{e2}
rA\cap sA=\emptyset\text{ for sufficiently large prime numbers }r\neq s,
\end{equation}
then \eqref{e1} holds true for every $f\in C(\T)$ such that ${\rm supp}\,\widehat f:=\{n\in\Z: \widehat{f}(n)\neq0\}=A$.
\item\label{H2}
Suppose that $A$ contains infinitely many primes.
Let $f\in C(\T)$ be such that ${\rm supp}\, \widehat{f}=A$ and all nonzero Fourier coefficients are positive.
Then~\eqref{e1} fails for $f$.
\end{enumerate}
\end{pr}
\begin{uw}
Note that every finite set satisfies \eqref{e2}, so all trigonometric polynomials satisfy~\eqref{e1}; the set $\{2^n: n\geq 1\}$
is an example of an infinite set satisfying~\ref{e2}.
\end{uw}
\begin{proof}[Proof of Proposition~\ref{p1}]
We will consider the behavior of the sums in~\eqref{e1} at $(0,0)$. Given $r,s$, two different prime numbers, we set
$I_{r,s}:=\{(x,y):  sx=ry\}$, which is a closed subgroup of $\T^2$ (of course $(0,0)\in I_{r,s}$), invariant under $T^r\times T^s$.
It is not hard to see that
\begin{equation}\label{e3} \frac1N\sum_{n\leq N} \delta_{rn\alpha,sn\alpha}\to \la_{I_{r,s}}\end{equation}
and that
\begin{align}
\begin{split}\label{e4}
&W\colon I_{r,s}\to\T,\;W(x,y)=ax+by\; (ar+bs=1)\\
&\mbox{is a continuous group isomorphism},
\end{split}
\end{align}
in particular, it sends $\la_{I_{r,s}}$ to $\la_{\T}$. In view of~\eqref{e3},
$$
\frac1N\sum_{n\leq N} f(T^{rn}0)\ov{f(T^{sn}0)}\to \int_{I_{r,s}}f\ot\ov{f}\,d\la_{I_{r,s}}.
$$
Since $W^{-1}(t)=(rt,st)$, it follows that
\begin{equation}\label{e5}
\int_{I_{r,s}}f\ot\ov{f}\,d\la_{I_{r,s}}=\int_{\T}(f\ot\ov{f})\circ W^{-1}(t)\,dt=\int_{\T}f(rt)\ov{f(st)}\,dt.\end{equation}
Now, if $f(t)=\sum_{n\in\Z}c_ne^{2\pi int}$ then
\begin{equation}\label{e6}
f(rt)=\sum_{n\in\Z}c_nne^{2\pi irnt},\;\;f(st)=\sum_{n\in\Z}c_ne^{2\pi isnt}\end{equation}
and we can compute $\int_{I_{r,s}}f\ot\ov{f}\,d\la_{I_{r,s}}$ using~\eqref{e5}, \eqref{e6} and Parseval's formula.

\eqref{H2} Fix infinitely many pairs $(r_j,s_j)\in A\times A$ of distinct prime numbers, $r_j,s_j\to\infty$. Take $j\geq 1$. It is enough to show that $\int_{I_{r_j,s_j}}f\ot\ov{f}\,d\la_{I_{r_j,s_j}}\neq0$.
By Parseval's formula, and the fact that all Fourier coefficients of $f$ are positive, we have $\int_{I_{r_j,s_j}}f\ot\ov{f}\,d\la_{I_{r_j,s_j}}\geq0$.
Since $r_js_j=s_jr_j$,  the $s_jr_j$th Fourier coefficient of $f(r_j\cdot)$ is $c_{s_j}>0$, while the $r_js_j$th Fourier coefficient of $f(s_j\cdot)$ is $c_{r_j}>0$.
Hence, by   Parseval's formula, $\int_{I_{r_j,s_j}}f\ot\ov{f}\,d\la_{I_{r_j,s_j}}>0$.

\eqref{H1} 
Each $(x,y)\in\T^2$ belongs to a coset of $I_{r,s}$. The proof for an arbitrary coset of $I_{r,s}$ goes along the same lines as for $I_{r,s}$ itself (on the cosets of  $I_{r,s}$, we consider translations of Haar measure $\la_{I_{r,s}}$ and $W$ is practically the same). Since $rA\cap   sA=\emptyset$,  the supports of the Fourier transforms of $f(r\cdot)$ and $f(s\cdot)$ are disjoint, whence  $\int_{\T}f(rt)\ov{f(st)}\,dt=0$.
\end{proof}

\footnotesize
\bibliography{cala.bib}

\vspace{2ex}
\noindent Joanna Ku\l aga-Przymus:\\
Institute of Mathematics, Polish Academy of Sciences, \'{S}niadeckich 8, 00-956 Warsaw, Poland \\
and
\\
Faculty of Mathematics and Computer Science, Nicolaus Copernicus University, Chopina 12/18, 87-100 Toru\'{n}, Poland\\
\textit{E-mail address:} \texttt{joanna.kulaga@gmail.com}\\
\\
\noindent Mariusz Lema\'nczyk:\\
Faculty of Mathematics and Computer Science, Nicolaus Copernicus University, Chopina 12/18, 87-100 Toru\'{n}, Poland\\
\textit{E-mail address:} \texttt{mlem@mat.umk.pl}\\

%\begin{minipage}[l]{.5\textwidth}
%\noindent Joanna Ku\l aga-Przymus:\\ \\
%Faculty of Mathematics\\
% and Computer Science,\\
%Nicolaus Copernicus University,\\
% Chopina 12/18,\\
%87-100 Toru\'{n}, Poland\\ \\
%and\\ \\
%Institute of Mathematics,\\
%Polish Acadamy of Sciences,\\
%\'{S}niadeckich 8,\\
%00-956 Warszawa, Poland\\
%\\
%joasiak@mat.umk.pl
%\end{minipage}
%\begin{minipage}[adjusting]{.5\textwidth}
%\noindent Mariusz Lema\'nczyk:\\
%\\
%Faculty of Mathematics\\
%and Computer Science,\\
%Nicolaus Copernicus University,\\
%Chopina 12/18,\\
%87-100 Toru\'{n}, Poland\\ \\
%mlem@mat.umk.pl\\
%\\
%\\
%\\
%\\
%\\
%\\
%\end{minipage}

\end{document}